\documentclass[12pt]{amsart}
\usepackage{amsmath,amssymb,amsbsy,amsfonts,latexsym,amsopn,amstext,
                                               amsxtra,euscript,amscd,bm}
\usepackage{url}
\usepackage[colorlinks,linkcolor=blue,anchorcolor=blue,citecolor=blue]{hyperref}
\usepackage{color}
\usepackage{float}

\hypersetup{breaklinks=true}

\begin{document}

\newtheorem{theorem}{Theorem}
\newtheorem{lemma}[theorem]{Lemma}
\newtheorem{claim}[theorem]{Claim}
\newtheorem{cor}[theorem]{Corollary}
\newtheorem{prop}[theorem]{Proposition}
\newtheorem{definition}{Definition}
\newtheorem{question}[theorem]{Question}
\newcommand{\hh}{{{\mathrm h}}}

\numberwithin{equation}{section}
\numberwithin{theorem}{section}
\numberwithin{table}{section}

\def\sssum{\mathop{\sum\!\sum\!\sum}}
\def\ssum{\mathop{\sum\ldots \sum}}
\def\iint{\mathop{\int\ldots \int}}

\def\squareforqed{\hbox{\rlap{$\sqcap$}$\sqcup$}}
\def\qed{\ifmmode\squareforqed\else{\unskip\nobreak\hfil
\penalty50\hskip1em\null\nobreak\hfil\squareforqed
\parfillskip=0pt\finalhyphendemerits=0\endgraf}\fi}%%

%  use the AMS-Euler Fraktur fonts
%%%%%%%%%%%%%%%%%%%%%%%%%%%%%%%%%%
\newfont{\teneufm}{eufm10}
\newfont{\seveneufm}{eufm7}
\newfont{\fiveeufm}{eufm5}
%%%%%%%%%%%%%%%%%%%%%%%%%%%%%%%%%
%
%  allow automatic size selection in math mode
%
%%%%%%%%%%%%%%%%%%%%%%%%%%%%%%%%%
\newfam\eufmfam
     \textfont\eufmfam=\teneufm
\scriptfont\eufmfam=\seveneufm
     \scriptscriptfont\eufmfam=\fiveeufm
%%%%%%%%%%%%%%%%%%%%%%%%%%%%%%%%%
%
%  \frak works on a single symbol at a time...
%
\def\frak#1{{\fam\eufmfam\relax#1}}

\newcommand{\bflambda}{{\boldsymbol{\lambda}}}
\newcommand{\bfmu}{{\boldsymbol{\mu}}}
\newcommand{\bfxi}{{\boldsymbol{\xi}}}
\newcommand{\bfrho}{{\boldsymbol{\rho}}}

\def\fK{\mathfrak K}
\def\fT{\mathfrak{T}}

\def\fA{{\mathfrak A}}
\def\fB{{\mathfrak B}}
\def\fC{{\mathfrak C}}

\def \balpha{\bm{\alpha}}
\def \bbeta{\bm{\beta}}
\def \bgamma{\bm{\gamma}}
\def \blambda{\bm{\lambda}}
\def \bchi{\bm{\chi}}
\def \bphi{\bm{\varphi}}
\def \bpsi{\bm{\psi}}

\def\eqref#1{(\ref{#1})}

\def\vec#1{\mathbf{#1}}

%\def\squareforqed{\hbox{\rlap{$\sqcap$}$\sqcup$}}
%\def\qed{\ifmmode\squareforqed\else{\unskip\nobreak\hfil
%\penalty50\hskip1em\null\nobreak\hfil\squareforqed
%\parfillskip=0pt\finalhyphendemerits=0\endgraf}\fi}

%%%%%%%%%%%%%%%%%%%%%%%%%
% Alphabet calligraphie %
%%%%%%%%%%%%%%%%%%%%%%%%%
\def\cA{{\mathcal A}}
\def\cB{{\mathcal B}}
\def\cC{{\mathcal C}}
\def\cD{{\mathcal D}}
\def\cE{{\mathcal E}}
\def\cF{{\mathcal F}}
\def\cG{{\mathcal G}}
\def\cH{{\mathcal H}}
\def\cI{{\mathcal I}}
\def\cJ{{\mathcal J}}
\def\cK{{\mathcal K}}
\def\cL{{\mathcal L}}
\def\cM{{\mathcal M}}
\def\cN{{\mathcal N}}
\def\cO{{\mathcal O}}
\def\cP{{\mathcal P}}
\def\cQ{{\mathcal Q}}
\def\cR{{\mathcal R}}
\def\cS{{\mathcal S}}
\def\cT{{\mathcal T}}
\def\cU{{\mathcal U}}
\def\cV{{\mathcal V}}
\def\cW{{\mathcal W}}
\def\cX{{\mathcal X}}
\def\cY{{\mathcal Y}}
\def\cZ{{\mathcal Z}}
\newcommand{\rmod}[1]{\: \mbox{mod} \: #1}

\def\cg{{\mathcal g}}

\def\vr{\mathbf r}

\def\e{{\mathbf{\,e}}}
\def\ep{{\mathbf{\,e}}_p}
\def\em{{\mathbf{\,e}}_m}

\def\Tr{{\mathrm{Tr}}}
\def\Nm{{\mathrm{Nm}}}

 \def\SS{{\mathbf{S}}}

\def\lcm{{\mathrm{lcm}}}

\def\({\left(}
\def\){\right)}
\def\fl#1{\left\lfloor#1\right\rfloor}
\def\rf#1{\left\lceil#1\right\rceil}

\def\mand{\qquad \mbox{and} \qquad}

         \newcommand{\comm}[1]{\marginpar{\vskip-\baselineskip
%raise themarginpar a bit
         \raggedright\footnotesize
\itshape\hrule\smallskip#1\par\smallskip\hrule}}

%%%%%%%%%%%%%%%%%%%%%%%%%%%%%%%%%%%%%%%%%%%%%%%%%%%%%%%%
%%%%%%%%%%%%%%%%%%%%%%%%%%%%%%%%%%%%%%%%%%%%%%%%%%%%%%%%
%%%%%%%%%%%%%%%%%%%%%%%%%%%%%%%%%%%%%%%%%%%%%%%%%%%%%%%%
%%%%%%%%%%%%%%%%%%%%%%%%%%%%%%%%%%%%%%%%%%%%%%%%%%%%%%%%

%%%%%%%  END OF STANDARD STUFF %%%%%%%%%

%%%%%%%%%%%%%%%%%%%%%%%%%%%%%%%%%%%%%%%%%%%%%%%%%%%%%%%%
%%%%%%%%%%%%%%%%%%%%%%%%%%%%%%%%%%%%%%%%%%%%%%%%%%%%%%%%
%%%%%%%%%%%%%%%%%%%%%%%%%%%%%%%%%%%%%%%%%%%%%%%%%%%%%%%%
%%%%%%%%%%%%%%%%%%%%%%%%%%%%%%%%%%%%%%%%%%%%%%%%%%%%%%%
%%%%%%%%%%%
%%% Spell

\hyphenation{re-pub-lished}

\mathsurround=1pt

\def\bfdefault{b}
\overfullrule=5pt

\def \F{{\mathbb F}}
\def \K{{\mathbb K}}
\def \Z{{\mathbb Z}}
\def \Q{{\mathbb Q}}
\def \R{{\mathbb R}}
\def \C{{\\mathbb C}}
\def\Fp{\F_p}
\def \fp{\Fp^*}

\def\Kmn{\cK_p(m,n)}
\def\psmn{\psi_p(m,n)}
\def\SI{\cS_p(\cI)}
\def\SIJ{\cS_p(\cI,\cJ)}
\def\SAIJ{\cS_p(\cA;\cI,\cJ)}
\def\SABIJ{\cS_p(\cA,\cB;\cI,\cJ)}
\def \xbar{\overline x_p}

\title[Cancellations Amongst Kloosterman Sums]{Cancellations Amongst Kloosterman Sums}

 \author[I. E. Shparlinski] {Igor E. Shparlinski}
 \thanks{This work was  supported in part by ARC Grant~DP140100118 (for I. E. Shparlinski), by NSFC Grant~11201275
and the Fundamental Research Funds for the Central Universities
Grant~GK201503014 (for T. P. Zhang)}
  \thanks{T. P. Zhang was the corresponding author (tpzhang@snnu.edu.cn).}

\address{Department of Pure Mathematics, University of New South Wales,
Sydney, NSW 2052, Australia}
\email{igor.shparlinski@unsw.edu.au}

 \author[T. P. Zhang] {Tianping Zhang}
\address{School of Mathematics and Information Science, Shaanxi Normal University, Xi'an 710062 Shaanxi, P. R. China}
\email{tpzhang@snnu.edu.cn}
%% \date{ }

\begin{abstract} We obtain several  estimates for bilinear form with Kloosterman sums. Such  results
can be interpreted as a measure of cancellations amongst with
parameters from short intervals. In particular, for certain ranges of parameters we improve
some recent results of Blomer,  Fouvry, Kowalski, Michel, and Mili{\'c}evi{\'c} (2014)
and Fouvry,   Kowalski and Michel (2014).
 \end{abstract}

\keywords{Kloosterman sums, cancellation, bilinear form}
\subjclass[2010]{11D79, 11L07}

\maketitle

\section{Introduction}
%\subsection{Background and motivation}
\label{sec:intro}

 Let  $p$ be   a sufficiently large
prime. For integers $m$ and $n$ we define
the Kloosterman sum
$$
\Kmn = \sum_{x=1}^{p-1} \ep\(mx +n\xbar \),
$$
where $\xbar$ is the multiplicative inverse of $x$ modulo $p$ and
$$
\ep(z) = \exp(2 \pi i z/p).
$$
Furthermore, given two intervals
$$
\cI = [K+1, K+M],\ \cJ = [L+1, L+N] \subseteq [1, p-1],
$$
and  two sequences of weights $\cA = \{\alpha_m\}_{m\in \cI}$ and $\cB = \{\beta_n\}_{n\in \cJ}$, we define the bilinear sums of Kloosterman sums
$$
\SABIJ = \sum_{m\in \cI} \sum_{n \in \cJ} \alpha_m \beta_n \cK_p(mn,1).
$$
We also consider the following special cases
 \begin{equation*}
\begin{split}
& \SAIJ = \cS_p\(\cA, \{1\}_{n=1}^N;\cI,\cJ\)= \sum_{m\in \cI} \sum_{n \in \cJ} \alpha_m \cK_p(mn,1),\\
&\SIJ= \cS_p\(\{1\}_{m=1}^M, \{1\}_{n=1}^N;\cI,\cJ\)= \sum_{m\in \cI} \sum_{n \in \cJ} \cK_p(mn,1),\\
&\SI = \cS_p\(\{1\}_{m=1}^M, \emptyset;\cI,\cJ\) =  \sum_{m\in \cI}  \cK_p(m,1).
\end{split}
\end{equation*}

Making the change of variable $x \to nx \pmod p$, one immediately observes  that $\cK_p(mn,1)  = \Kmn$, thus  we also have
$$
 \SIJ= \sum_{m\in \cI} \sum_{n \in \cJ}  \Kmn.
$$

We aslo define,  for real $\sigma>0$,
 $$
 \|\cA\|_\sigma =\( \sum_{m\in \cI} |\alpha_m|^\sigma\)^{1/\sigma}
 \mand
  \|\cB\|_\sigma =\( \sum_{n\in \cJ} |\beta_n|^\sigma\)^{1/\sigma}
 $$
 with the usual convention
 $$
 \|\cA\|_\infty=\max_{m\in \cI}|\alpha_m| \mand  \|\cB\|_\infty=\max_{n\in \cJ}|\beta_n|.
$$

By the Weil bound we have
$$
|\Kmn|\le 2 p^{1/2},
$$
see~\cite[Theorem~11.11]{IwKow}. Hence
 \begin{equation}
\label{eq:trivial AB}
\SABIJ \le 2 \|\cA\|_1 \|\cB\|_1 p^{1/2}.
\end{equation}
We are interested in studying cancellations amongst Kloosterman sums and
thus improvements of the trivial bound~\eqref{eq:trivial AB}.

Throughout the paper,  as usual $A\ll B$  is equivalent to the inequality $|A|\le cB$
with some  constant $c>0$
(all implied constants are absolute throughout the paper).

\section{Previous results}

First we note that by a very special case of a much more general
result of Fouvry, Michel, Rivat and S{\'a}rk{\"o}zy~\cite[Lemma~2.3]{FMRS} we have
$$
\SI  \ll p \log p
$$
which for  $M \ge p^{1/2} \log p$
 improves the trivial bound $\SI \le 2M p^{1/2}$ following from~\eqref{eq:trivial AB}.
 Recently,   Fouvry,  Kowalski,
 Michel,   Raju,   Rivat and   Soundararajan~\cite[Corollary~1.6]{FKMRRS}
have given the bound
$$
\SI  \ll M p^{1/2} (\log p)^{-\eta}
$$
provided that $M \ge p^{1/2}  (\log p)^{-\eta}$ with some absolute
constant $\eta > 0$.

It is also easy to derive from~\cite[Theorem~7]{Shp1} that
$$
 \SIJ  \ll MNp^{1/4} + M^{1/2} N^{1/2} p^{1+o(1)},
$$
which improves the trivial bound from~\eqref{eq:trivial AB} for
$MN \ge p^{1+ \varepsilon}$ for any fixed $\varepsilon > 0$.

The sums $\SABIJ$ and $\SAIJ$
 have been estimated by Fouvry,   Kowalski and Michel~\cite[Theorem~1.17]{FKM}
 as  a part of a much more general result about sums of so-called {\it trace functions\/}.
Then, by~\cite[Theorem~1.17(2)]{FKM}, for initial intervals
 $\cI = [1,M]$ and $\cJ = [1,N]$,
 we have
 \begin{equation}
\label{eqSAIJ-1}
 \SAIJ  \le  \|\cA\|_1 p^{1+o(1)} .
 \end{equation}
 Furthermore, by a result of Blomer,  Fouvry, Kowalski, Michel, and Mili{\'c}evi{\'c}~\cite[Theorem~6.1]{BFKMM},
 also for an initial interval $\cI$ and an arbitrary interval $\cJ$, with
 $$
 MN \le p^{3/2} \mand M \le N^2
 $$
 we have
 \begin{equation}
\label{eqSAIJ-2}
 \SAIJ  \le  \( \|\cA\|_1 \|\cA\|_2\)^{1/2}  M^{1/12} N^{7/12} p^{3/4+o(1)}.
 \end{equation}
%In particular,
%we can simplify~\eqref{eqSAIJ-1} and~\eqref{eqSAIJ-2} as
% \begin{equation}
%\label{eqSAIJ1-simpl}
% \SAIJ  \ll   \|\cA\|_\infty M p^{1+o(1)}
%\end{equation}
%and
% \begin{equation}
%\label{eqSAIJ2-simpl}
% \SAIJ  \ll  \|\cA\|_\infty M^{5/6} N^{7/12} p^{1/4+o(1)},
%\end{equation}
%respectively.
%
One can also find in~\cite{BFKMM,FKM,KMS} a series of other bounds
on the sums $ \SAIJ$ and $\SABIJ$ and also on more general sums.

Finally, Khan~\cite{Khan} has given a nontrivial estimate
for the analogue of $\SI$ modulo a fixed prime power which is nontrivial
already for $M  \ge p^{\varepsilon}$.

\section{New results}

We start with the sums $\SIJ$ and present a bound which improves~\eqref{eq:trivial AB} already for  $MN \ge p^{1/2+\varepsilon}$.

\begin{theorem}
\label{thm:SIJ}  We have,
$$
 \SIJ  \ll  (p+MN)p^{o(1)}.
$$
\end{theorem}

We now estimate $\SAIJ$.

\begin{theorem}
\label{thm:SAIJ}  We have,
$$
 \SAIJ  \ll   \|\cA\|_2 N^{1/2}p.
$$
\end{theorem}

We can re-write the bounds~\eqref{eqSAIJ-1} and~\eqref{eqSAIJ-2}
in terms of the $\|\cA\|_\infty$ as
 \begin{equation}
\label{eq:SAIJ1-simpl}
 \SAIJ  \ll   \|\cA\|_\infty M p^{1+o(1)}
\end{equation}
and
 \begin{equation}
\label{eq:SAIJ2-simpl}
 \SAIJ  \ll  \|\cA\|_\infty M^{5/6} N^{7/12} p^{3/4+o(1)},
\end{equation}
respectively, and the bound of Theorem~\ref{thm:SAIJ} as
 \begin{equation}
\label{eq:SAIJnew-simpl}
 \SAIJ  \ll  \|\cA\|_\infty M^{1/2} N^{1/2} p^{1+o(1)}.
\end{equation}
We now see for any fixed $\varepsilon>0$ the bound~\eqref{eq:SAIJnew-simpl}
improves~\eqref{eq:SAIJ1-simpl} and~\eqref{eq:SAIJ2-simpl}  for
$$
N < M p^{-\varepsilon} \mand M^4 N \ge p^{3 + \varepsilon}
$$
respectively,
and also applies to  intervals $\cI$ and $\cJ$ at arbitrary positions.

\section{Preparations}

We need the following simple result.
\begin{lemma}
\label{lem:Invers} For any
integers  %$A$,
$X$ and $Y$ with $1\le X,Y < p$,
%% $0\le B < B+L < p$  and $0 \le M < p$,
the congruence
$$
%%(A+x)
xy \equiv 1 \pmod p, \qquad 1 \le |x| \le X, \ 1 \le |y| \le Y
$$
has at most $(X Y/p + 1) p^{o(1)}$ solutions.
\end{lemma}

\begin{proof} Writing $xy \equiv 1 \pmod p$ as $xy = 1 + kp$ for some
integer $k$ with $|k| \le XY/p$ and using the bound on the divisor function,
see~\cite[Equation~(1.81)]{IwKow},
we get the desired estimate.
\end{proof}

We also  need the following well-known result, which dates back to
 Vinogradov~\cite[Chapter 6, Problem 14.a]{Vin}.

\begin{lemma}
\label{lem:DoubleSum} For arbitrary set $\cU, \cV \subseteq \{0, \ldots, p-1\}$ and complex
numbers $\varphi_u$ and $\psi_v$ with
$$
\sum_{u \in \cU} |\varphi_u|^2  \le \varPhi  \mand
\sum_{v \in \cV}|\psi_v|^2 \le \varPsi,
$$
we have
$$
\left|
\sum_{u\in \cU} \sum_{v \in \cV}
\varphi_u \psi_v \ep(uv) \right| \le \sqrt{\varPhi\varPsi  p}.
$$
\end{lemma}

\section{Proof of Theorem~\ref{thm:SIJ}}

For an integer $u$ we define
$$
\|u\|_p = \min_{k \in \Z} |u - kp|
$$
as the distance to the closest integer.
Then, changing the order of summation,  we obtain
$$
 \SIJ \ll \sum_{x=1}^{p-1}\min\left\{M,  \frac{p}{\|x\|_p}\right \}
  \min\left\{N,  \frac{p}{\|\xbar\|_p}\right \},
$$
see~\cite[Bound~(8.6)]{IwKow}.
We now write
 \begin{equation}
\label{eq:BIJ S14}
 \SIJ \ll  MN S_1+MpS_2+Np S_3 + p^2 S_4,
\end{equation}
where
 \begin{equation*}
\begin{split}
S_1 & = \sum_{\substack{x=1\\\|x\|_p \le p/M\\ \|\xbar\|_p \le p/N}}^{p-1} 1, \qquad S_2 = \sum_{\substack{x=1\\\|x\|_p \le p/M\\ \|\xbar\|_p > p/N}}^{p-1}  \frac{1}{\|\xbar\|_p},\\
S_3&  = \sum_{\substack{x=1\\\|x\|_p > p/M\\ \|\xbar\|_p  \le p/N}}^{p-1}  \frac{1}{\|x\|_p}, \qquad
S_4 = \sum_{\substack{x=1\\\|x\|_p > p/M\\ \|\xbar\|_p > p/N}}^{p-1}  \frac{1}{\|x\|_p\|\xbar\|_p}.
\end{split}
\end{equation*}

By Lemma~\ref{lem:Invers} we immediately obtain
 \begin{equation}
\label{eq:S1}
S_1 \le  (p/MN + 1)p^{o(1)}.
\end{equation}

To estimate $S_2$, we define $I = \rf{\log p}$ and  write
$$
S_2 \le \sum_{i =0}^{I} S_{2,i},
$$
where
\begin{equation*}
\begin{split}
S_{2,i} & = \sum_{\substack{x=1\\\|x\|_p \le p/M\\ e^{i+1} p/N\ge \|\xbar\|_p > e^ip/N}}^{p-1}  \frac{1}{\|\xbar\|_p}\\
&\ll e^{-i} Np^{-1}  \sum_{\substack{x=1\\\|x\|_p \le p/M\\ e^{i+1} p/N\ge \|\xbar\|_p > e^ip/N}}^{p-1} 1
\ll e^{-i} Np^{-1}  \sum_{\substack{x=1\\\|x\|_p \le p/M\\ \|\xbar\|_p\le e^{i+1} p/N}}^{p-1} 1.
%\\
%&\ll e^{-i} Np^{-1}(e^{i+1}p/MN+1)p^{o(1)}
%\ll (e/M+e^{-i} Np^{-1})p^{o(1)}.
\end{split}
\end{equation*}

Now we use  Lemma~\ref{lem:Invers} again to derive
 \begin{equation}
\label{eq:S2}
\begin{split}
S_2 &\le \sum_{i=0}^{I}\(M^{-1}+e^{-i} Np^{-1}\)p^{o(1)}\\
&\ll (I+1)M^{-1}p^{o(1)}+ Np^{-1+o(1)} \sum_{i=0}^{I} e^{-i}\\
&\ll (I+1) M^{-1}p^{o(1)}+ Np^{-1+o(1)} \\
&\ll M^{-1} p^{o(1)}+ Np^{-1+o(1)}.
\end{split}
\end{equation}

Similarly we obtain
 \begin{equation}
\label{eq:S3}
S_3 \le N^{-1}p^{o(1)}+ Mp^{-1+o(1)}.
\end{equation}

Finally, we write
$$
S_4 \le \sum_{i,j=0}^{I} S_{4,i,j},
$$
where
\begin{equation*}
\begin{split}
S_{4,i,j} & = \sum_{\substack{x=1\\ e^{i+1} p/M \ge \|x\|_p > e^i p/M\\ e^{j+1} p/N\ge \|\xbar\|_p > e^jp/N}}^{p-1}   \frac{1}{\|x\|_p\|\xbar\|_p}\\
&\ll e^{-i-j} MNp^{-2}  \sum_{\substack{x=1\\ e^{i+1} p/M \ge \|x\|_p > e^i p/M\\ e^{j+1} p/N\ge \|\xbar\|_p > e^jp/N}}^{p-1} 1\\
&\ll e^{-i-j} MNp^{-2}  \sum_{\substack{x=1\\\|x\|_p \le  e^{i+1} p/M\\ \|\xbar\|_p\le e^{j+1} p/N}}^{p-1} 1.
\end{split}
\end{equation*}
Applying  Lemma~\ref{lem:Invers} one more time, we obtain
\begin{equation*}
\begin{split}
S_{4,i,j} &  \ll e^{-i-j} MNp^{-2} \(e^{i+j} p/MN + 1\)p^{o(1)}\\
&= (p^{-1}+e^{-i-j}MNp^{-2})p^{o(1)}.
\end{split}
\end{equation*}
Hence
 \begin{equation}
\label{eq:S4}
\begin{split}
S_4 &\le \sum_{i,j=0}^{I}(p^{-1}+e^{-i-j}MNp^{-2})p^{o(1)}\\
&\le (I+1)^2 p^{-1+o(1)}+MNp^{-2+o(1)}\\
&\le p^{-1+o(1)}+MNp^{-2+o(1)}.
\end{split}
\end{equation}

Combining~\eqref{eq:S1}, \eqref{eq:S2}, \eqref{eq:S3} and~\eqref{eq:S4}
 we obtain the result.

\section{Proof of Theorem~\ref{thm:SAIJ}}

Changing the order of summation, as in the proof of Theorem~\ref{thm:SIJ}  we obtain
 \begin{equation*}
\begin{split}
 \SAIJ  & = \left|\sum_{m\in \cI} \sum_{x=1}^{p-1}\alpha_m
 \gamma_x  \ep(m \xbar)\right|,
\end{split}
\end{equation*}
%& \ll \left|\sum_{m\in \cI} \sum_{x=1}^{p-1}\alpha_m\min\left\{N,  \frac{p}{\|\xbar\|_p}\right \} \ep(m x)\right|\\
where
$$
|\gamma_x| \le  \min\left\{N,  \frac{p}{\|x\|_p}\right \}.
$$

Thus, similarly to the  proof of Theorem~\ref{thm:SIJ}  we define $I = \rf{\log p}$ and write
 \begin{equation}
\label{eq:BAIJ S0i}
 \SIJ \ll  |S_0|+ \sum_{i=1}^I |S_i|,
\end{equation}
where
 \begin{equation*}
\begin{split}
S_0 &   = \sum_{m\in \cI} \sum_{\substack{x=1\\ \|x\|_p \le p/N}}^{p-1}  \alpha_m
 \gamma_x  \ep(m \xbar), \\
S_i & =  \sum_{m\in \cI} \sum_{\substack{x=1\\  e^{i+1} p/N\ge \|x\|_p > e^ip/N}}^{p-1}
\alpha_m   \gamma_x  \ep(m \xbar), \qquad i = 1, \ldots, I.
\end{split}
\end{equation*}
%
%Then we have
%\begin{equation*}
%\begin{split}
%S_0 &   = \sum_{m\in \cI}\sum_{v\in \cW}  \alpha_m
% \gamma_v  \ep(m v), \\
%S_i & =  \sum_{m\in \cI}\sum_{v\in \cW_i}  \alpha_m
% \gamma_v  \ep(m v), \qquad i = 1, \ldots, I,
%\end{split}
%\end{equation*}
%where
% \begin{equation*}
%\begin{split}
%\cW_0 &   = \left\{\xbar: \|x\|_p \le p/N\right\}, \\
%\cW_i &   = \left\{\xbar: e^{i+1} p/N\ge \|x\|_p > e^{i} p/N\right\}, \quad  i = 1, \ldots, I.
%\end{split}
%\end{equation*}

%{\bf Now use L.~\ref{lem:DoubleSum}}
 Now use Lemma~\ref{lem:DoubleSum}, we have
\begin{equation}
\label{eq:S0}
|S_0| \ll  \|\cA\|_2 N\sqrt{ (p/N+1) p }\le   \|\cA\|_2 N^{1/2}p.
\end{equation}
Also, for $i = 1, \ldots, I$, using that if  $e^{i+1} p/N\ge \|x\|_p > e^ip/N$
then $\gamma_x \ll Ne^{-i}$, hence, by  Lemma~\ref{lem:DoubleSum}, we obtain
 \begin{equation*}
\begin{split}
S_i & =  \sum_{m\in \cI} \sum_{\substack{x=1\\  e^{i+1} p/N\ge \|x\|_p > e^ip/N}}^{p-1}
\alpha_m   \gamma_x  \ep(m \xbar) \\
& \ll  \|\cA\|_2 (N^2e^{-2i} e^{i} p/N)^{1/2} p^{1/2} = e^{-i/2}  \|\cA\|_2 N^{1/2} p .
\end{split}
\end{equation*}
Therefore,
\begin{equation}
\label{eq:Si}
\sum_{i=1}^I |S_i| \ll   \|\cA\|_2 N^{1/2} p  \sum_{i=1}^I e^{-i/2}  \ll  \|\cA\|_2 N^{1/2} p.
\end{equation}
Combining~\eqref{eq:S0} and~\eqref{eq:Si},
we obtain the result.

\section{Comments}

It is also natural to consider cancellations between some other exponential and character sums.
For example, in~\cite{Shp2} one can find some bound on  the following sums
\begin{equation*}
\begin{split}
&\cS_{p}(f,\cA, \cB; \fC)=
\sum_{(u,v)\in \fC} \alpha_u \beta_v \e_p(v/f(u)),
\\
&\cT_p(f,\cA, \cB; \fC)=
\sum_{(u,v)\in \fC} \alpha_u \beta_v \chi(v+f(u)),
\end{split}
\end{equation*}
(where $\chi$ is a multiplicative character modulo $p$),
over a {\it convex\/} set
$\fC \subseteq [1,U]\times [1,V]$,
with some integer $1 \le U,V < p$.

Here we also note that one can also obtain a nontrivial cancellation for sums
$$
\cH_{k,p}(a;\cI) = \sum_{m\in \cI} \cG_{k,p}(am)
$$
of Gaussian sums
$$
\cG_{k,p}(a) =  \sum_{x=0}^{p-1}\ep\(ax^k\)
$$
with a positive integer $k \mid p-1$.
Indeed, we define
$$
 \tau_p(a;\chi)= \sum_{x=1}^{p-1}\chi(x) \ep\(ax\),
$$
where $\chi$ is a multiplicative
character; we refer to~\cite[Chapter~3]{IwKow} for a background on multiplicative
characters.  Then by the orthogonality of characters, we have
$$
\cG_{k,p}(a) = \sum_{\substack{\chi^k = \chi_0\\\chi\ne \chi_0}}\overline \chi(a) \tau_p(a;\chi),
$$
where the summation is over all nonprincipal multiplicative characters $\chi$ modulo $p$, such that $\chi^k$
is the principal character $\chi_0$, see also~\cite[Theorem~5.30]{LN}.
Using that $|\tau_p(a;\chi)| = p^{1/2}$ for any nonprincipal multiplicative characters $\chi$
and integer $a$ with $\gcd(a,p)=1$, we derive
$$
|\cH_{k,p}(a;\cI) |= p^{1/2}  \sum_{\substack{\chi^k = \chi_0\\\chi\ne \chi_0}}\overline \chi(a)
 \sum_{m\in \cI} \overline \chi(m).
 $$
 Thus applying the  Burgess bound, see~\cite[Equation~(12.58)]{IwKow},
 we derive that
\begin{equation}
\begin{split}
\label{eq:Ha}
\cH_{k,p}(a;\cI)  &\ll    M^{1 -1/\nu} p^{1/2 + (\nu+1)/(4\nu^2)}(\log p)^{1/\nu} \\
& =    M^{1 -1/\nu} p^{(2\nu^2 + \nu+1)/(4\nu^2)}(\log p)^{1/\nu}
\end{split}
\end{equation}
for any fixed $k\mid p-1$ and $\nu = 1,2, \ldots$.

Similarly, for general quadratic polynomials $f(X) = aX^2 + bX$, with $\gcd(a,p)=1$, we can define
the double sums
$$
\cF_{p}(a,b;\cI) = \sum_{m\in \cI} \sum_{x=0}^{p-1} \ep\(m\(ax^2 + bx\)\).
$$
It is easy to see that
\begin{align*}
\sum_{x=0}^{p-1} &\ep(ax^2 + bx)  =
 \ep\(-\frac{b^2}{4a}\)\sum_{x=0}^{p-1}  \ep\(a\(x+ b/(2a)\)^2 \) \\
& = \ep\(-\frac{b^2}{4a}\)\sum_{x=0}^{p-1}  \ep\(ax^2 \) = \(\frac{a}{p}\)   \ep\(-\frac{b^2}{4a}\)
\cG_{2,p}(1)   ,
\end{align*}
where $(a/p)$ is the Legendre symbol of $a$ modulo $p$.  Hence
\begin{align*}
\sum_{m \in \cI}
\sum_{x=0}^{p-1} \ep\(m(ax^2 + bx)\)
& =  \cG_{2,p}(1)    \sum_{m \in \cI}  \(\frac{am}{p}\)
 \ep\(-\frac{(bm)^2}{4am}\)\\
&= \(\frac{a}{p}\)
\cG_{2,p}(1)   \sum_{m \in \cI}  \(\frac{m}{p}\)  \ep\(-\frac{b^2 m}{4a}\).
 \end{align*}
Now, using  the bound of Burgess~\cite{Burg}  on
short mixed sums (see~\cite{HBP,Kerr,Pierce} for
various generalisations) we easily derive that
for any fixed  $\nu = 2, 3, \ldots$ we have
\begin{equation}
\label{eq:Fab}
\begin{split}
\cF_{p}(a,b;\cI)& \ll   M^{1 -1/\nu} p^{1/2 + 1/(4(\nu-1))} (\log p)^2\\
& = M^{1 -1/\nu} p^{(2\nu-1)/(4(\nu-1))} (\log p)^2,
\end{split}
\end{equation}
where the implied constant may depend on $\nu$.

We note that the bounds~\eqref{eq:Ha} and~\eqref{eq:Fab} are nontrivial
provided that $M \ge p^{1/4+\varepsilon}$ for any fixed $\varepsilon>0$
and sufficiently large $p$.

\section*{Acknowledgement}

%\comm{Add grant}
The second author gratefully acknowledge the support, the hospitality
and the excellent conditions at the School of Mathematics and Statistics of UNSW during his visit.

This work was  supported in part by ARC Grant~DP140100118 (for I. E. Shparlinski), by NSFC Grant~11201275
and the Fundamental Research Funds for the Central Universities
Grant~GK201503014 (for T. P. Zhang).

\end{document}